\documentclass[a4paper,12pt]{amsart} 

\usepackage{amsmath,amsthm,amssymb} 
\usepackage[all, warning]{onlyamsmath} 
\usepackage{bm} 
\usepackage{enumerate} 
\usepackage[pdftex]{graphicx}
\usepackage{url}
\RequirePackage[l2tabu, orthodox]{nag} 

\theoremstyle{plain} 
\newtheorem{theorem}{\indent Theorem}[section] 
\newtheorem{lemma}[theorem]{\indent Lemma}

\newtheorem{proposition}[theorem]{\indent Proposition}
\newtheorem*{theorem0}{\indent Theorem} 
\theoremstyle{definition} 

\newtheorem{example}[theorem]{\indent Example}
\theoremstyle{remark} 

\makeatletter
    
    \@addtoreset{equation}{section}
\makeatother

\newcommand{\Z}{{\mathbb{Z}}} 
\newcommand{\R}{{\mathbb{R}}} 
\newcommand{\C}{{\mathbb{C}}} 
\newcommand{\E}{{\mathbb{E}}} 
\newcommand{\F}{{\mathbb{F}}} 
\newcommand{\K}{{\mathbb{K}}} 
\newcommand{\M}{{\mathbb{M}}} 
\renewcommand{\P}{{\mathbb{P}}} 
\newcommand{\cA}{{\mathcal{A}}} 
\newcommand{\cB}{{\mathcal{B}}} 
\newcommand{\cC}{{\mathcal{C}}} 
\newcommand{\cF}{{\mathcal{F}}} 
\newcommand{\cI}{{\mathcal{I}}} 
\newcommand{\cM}{{\mathcal{M}}} 
\newcommand{\cN}{{\mathcal{N}}} 
\newcommand{\diam}{{\operatorname{diam}\,}}
\newcommand{\id}{{\operatorname{id}}}
\newcommand{\inte}{{\operatorname{int}\,}}
\newcommand{\norm}[1]{\Vert#1\Vert} 

\def\overA{\overline{A}}
\def\underA{\underline{A}}

\begin{document}
\title[limit theorem for PDs of MPPs]
{A Limit theorem for persistence diagrams of random filtered complexes built over marked point processes}
\author[T. Shirai]{Tomoyuki Shirai}
\author[K. Suzaki]{Kiyotaka Suzaki}


\address{
  Institute of Mathematics for Industry \endgraf
  Kyushu University \endgraf
  744, Motooka, Nishi-ku, \endgraf
  Fukuoka, 819-0395, Japan
}
\email{shirai@imi.kyushu-u.ac.jp}

\address{
  Mathematical Science Education Center \endgraf
  Headquarters for Admissions and Education \endgraf
  Kumamoto University \endgraf
  2-40-1 Kurokami, Chuo-ku, \endgraf
  Kumamoto, 860-8555, Japan
}
\email{k-suzaki@kumamoto-u.ac.jp}

\begin{abstract}
  We consider random filtered complexes built over marked point processes on Euclidean spaces.
  Examples of our filtered complexes include
  a filtration of $\check{\textrm{C}}$ech complexes
  of a family of sets with various sizes, growths, and shapes.
  We establish the law of large numbers for persistence diagrams
  as the size of the convex window observing a marked point process tends to infinity.
\end{abstract}

\maketitle

\section{Introduction}

Much attention has been paid to topological data analysis
(TDA) over the last few decades and
persistent homology has been playing a central role as
one of the most important tools in TDA.
Persistent homology measures persistence of topological
feature, in particular, appearance and disapperance
of homology generators in each dimension and enables us
to view data sets in multi-resolutional way.
There are several aspects to be discussed in the theory of
persistent homology, among which we focus on the random
aspect.
Data sets to be analyzed are often represented as binomial
processes if each data point is regarded as a sample from a certain
probability distribution and as stationary point processes if
data points are considered as part of a huge object.
There have been many works on the topology of binomial processes
from the viewpoint of manifold learning \cite{ATT, BM, BO}.
In the setting of stationary point processes,
Yogeshwaran-Adler \cite{YA} discussed the topology of random
complexes built over stationary point processes in the
Euclidean space and showed the strong law of large numbers for
Betti numbers of such random complexes.
In the same setting, Hiraoka-Shirai-Trinh \cite{HST} proved the
strong law of large numbers for persistence diagrams, which
comprise all information about persistence Betti numbers, and
also discussed the positivity of its limiting persistence diagram.
In the present paper, we extend the framework to deal with
random filtered complexes built over stationary \textit{marked} point
processes in order to include more natural examples such as
\textit{weighted} complexes (\cite{BLMRS}, \cite{BCOSS}, \cite{MALWX}, and references therein).

Given data as a finite point configuration $\Xi$ in
$\R^{d}$, we consider the union of closed balls
$\cup_{x\in \Xi}\overline{B}_{t}(x)$ of radius $t\ge 0$
centered at each data point $x\in \Xi$, which we denoted by
$X(\Xi, t)$. We are interested in
how the $q$-dimensional homology classes of
$X(\Xi, t)$ behave as $t$ grows.
By the so-called Nerve theorem, it is well-known that
$X(\Xi, t)$ is homotopy
equivalent to the $\check{\textrm{C}}$ech complex $C(\Xi,
  t)$, which is defined
as a simplicial complex over points in $\Xi$
consisting of $q$-simplices $\sigma=\{x_{0}, x_{1},
  \dots, x_{q}\}$ for which $\cap_{i=0}^{q}\overline{B}_{t}(x_{i})\neq \emptyset$.
We thus obtain a filtration of simplicial complexes
$C(\Xi)=\{C(\Xi, t)\}_{t\ge 0}$ from $\Xi$.  The $q$th
persistent homology of the filtration $C(\Xi)$
gives more topological information of data than
the homologies of snapshots of $C(\Xi, t)$
(cf. \cite{ELZ} and \cite{ZC}).

When we look at an atomic configuration,
it is natural to consider the influence of atomic radii.
In the usual setting, as explained above, we start from a finite set
of points in $\Xi$ and attach balls of radius $t$ to
construct the $\check{\textrm{C}}$ech complex, however,
taking atomic radii into account, it would be natural to
start from  a finite set of balls with initial radii rather
than a finite set of points.
If points are considered to have different shapes,
it would be better to attach
a different shape of $V_t(x_i)$ to the $i$th point
instead of a ball $B_t(x_i)$ depending on the shape of the
$i$th point.
In many applications, each point often has some extra information
and so we would like to incorporate it in our framework.
For this purpose, in the present paper,
we introduce a filtration of simplicial complexes
built over finite sets on $\R^{d}$ with \textit{marks} in a complete separable metric space $\M$.
Here by marks we mean additional information of data and
several information at each point can be expressed as a
mark by taking $\M$ appropriately.

Now we introduce some notations to state our main theorems.
We say that a nonempty finite subset $\Xi$ of $\R^{d}\times \M$ is a simple marked point set if
$\#(\Xi\cap \pi^{-1}\{x\})\le 1$ holds for any
$x\in \R^{d}$, where $\#A$ is the cardinality of a set $A$
and $\pi\,:\,\R^{d}\times \M\to \R^{d}$ is the natural
projection.
For a simple marked point set $\Xi$,
by forgetting marks by $\pi$, we obtain a simple
point set $\Xi_{g} = \pi(\Xi) \subset \R^d$ as
the ground point set of $\Xi$.
For a given simple marked point set $\Xi$,
we define a filtration of simplicial complexes
$\K(\Xi)=\{K(\Xi, t)\}_{t\ge 0}$
with the vertex sets in
$\Xi_{g}$ by assigning the
birth time $\kappa(\sigma)$ for each simplex
$\sigma_{g}=\pi(\sigma)\subset \Xi_{g}$, that is,
$K(\Xi, t)=\{\sigma_{g}\subset
  \Xi_{g}\,:\,\kappa(\sigma)\le t\}$, where
$\kappa$ is a function defined on the nonempty finite
subsets of $\R^{d}\times \M$ with some appropriate
conditions (see Section~\ref{subsec:kappa}).
We call $\K(\Xi)$ the $\kappa$-filtered
complex built over a simple marked point set $\Xi$.
This is a marked-version of $\kappa$-complex
(resp. filtration) introduced in
\cite{HST} as a generalization of $\check{\textrm{C}}$ech
and Vietoris-Rips complex (resp. filtration).
For example, if we consider the case where $\M$ is the
closed interval $[0, R]$ and
\[
  \kappa(\sigma)
  =\inf_{w\in \R^{d}}\max_{(x, r)\in \sigma}(\norm{x-w}-r)^{+} \text{ for finite } \sigma\subset \R^{d}\times \M,
\]
then $\K(\Xi, t)$ is the $\check{\textrm{C}}$ech complex of the family of closed balls
$\{\overline{B}_{t+r}(x)\}_{(x, r)\in \Xi}$,
which is the case where we start from balls with
various initial radii (Example~\ref{ex.K1}).
Thus our framework enables us to consider a filtration of
$\check{\textrm{C}}$ech complexes of a family of balls with various sizes naturally.
Later, we give several examples of $\kappa$-filtered complexes, which include a filtration of $\check{\textrm{C}}$ech complexes
of a family of sets with various growth speeds (Example \ref{ex of kappa 2}) and various shapes (Example \ref{ex of kappa 3}).
These examples are often called weighted complexes.

For a $\kappa$-filtered complex $\K(\Xi)$, its $q$th persistence diagram
\[
  D_{q}(\K(\Xi))
  =\{(b_{i}, d_{i})\in \Delta\,:\,i=1,2, \dots, n_{q}\}
\]
is defined by a multiset on $\Delta=\{(x, y)\in [0, \infty]\times[0, \infty]\,:\,x<y\}$ determined by the decomposition of the persistent homology (see Section  \ref{subsec:PH and PD}), which is an expression of the $q$th persistent homology.
Each $(b_{i}, d_{i})$ means that a $q$th homology class appears at $t=b_{i}$,
persists for $b_{i}\le t<d_{i}$, and disappears at $t=d_{i}$ in
$\K(\Xi)$.
In this paper, the persistence
diagram $D_{q}(\K(\Xi))$ is treated as the counting measure
\[
  \xi_{q}(\K(\Xi))=\sum_{(b, d)\in D_{q}(\K(\Xi))}\delta_{(b, d)},
\]
where $\delta_{(x, y)}$ denotes the Dirac measure at $(x, y)\in \Delta$.
Let $\Phi$ be a marked point process on $\R^{d}$ with marks in $\M$. It is a point process on $\R^{d}\times \M$ such that the ground process
$\Phi_{g}(\cdot)=\Phi(\cdot\times\M)$ is a
simple point process on $\R^{d}$. We assume that $\Phi_{g}$
has all finite moment, that is, $\E[\Phi_{g}(A)^{p}]<+\infty$ for any bounded Borel set $A$ in $\R^{d}$ and $p\ge 1$. The restricted marked point process $\Phi$ on $A\times\M$ is denoted by $\Phi_{A}$.
We discuss the strong law of large numbers for persistence diagrams of a random $\kappa$-filtered complex built over a marked point process, or more precisely, the asymptotic behavior of $\xi_{q}(\K(\Phi_{A_{n}}))$ ($\xi_{q, A_{n}}$ for short) of the $\kappa$-filtered complex $\K(\Phi_{A_{n}})=\{K(\Phi_{A_{n}}, t)\}_{t\ge 0}$ as the size of window $A_{n}$ tends to infinity, where $\{A_{n}\}_{n\in \cN}$ is an increasing net of bounded convex sets in
$\R^{d}$ with $\sup\{r>0\,:\,A_{n} \text{ contains a ball of radius } r\}\to \infty$ as $n\to \infty$.
Such a net is called a convex averaging net in $\R^{d}$.
The main purpose of this paper is to show the following.

\begin{theorem0}\label{th:main}
  Let $\Phi$ be a stationary ergodic marked point process and suppose its ground process $\Phi_{g}$ has all finite moments.
  Then for any nonnegative integer $q$,
  there exists a Radon measure $\nu_{q}$ on $\Delta$
  such that for any convex averaging net $\cA=\{A_{n}\}_{n\in\cN}$ in $\R^{d}$,
  \[
    \frac{1}{\ell(A_{n})}\xi_{q, A_{n}}\xrightarrow{v} \nu_{q}
    \quad \text{as} \quad n\to \infty \text{ a.s. }
  \]
  where $\ell$ is the $d$-dimensional Lebesgue measure and $\xrightarrow{v}$ denotes
  the vague convergence of measures on $\Delta$.
\end{theorem0}

New feature of this theorem is two-fold: marks and averaging nets.
The same limit theorem above for persistence diagrams is first established in \cite[Theorem $1.5$]{HST} for the case where $\{A_{n}\}_{n\in \cN}$ is the rectangles $\{[-L/2,
  L/2)^{d}\}_{L>0}$ and stationary ergodic point processes (without marks) on $\R^{d}$.
Marked point processes are often useful from application point of view (cf. \cite{BB, CSKM}) so that this extension greatly expanded the scope of application in TDA.
The limit theorem along convex averaging sequences can also be found in a recent article \cite{SW} when the underlying filtered complexes are basically $\check{\textrm{C}}$ech complexes. Our theorem is also an extension of \cite{SW} to the case of the class of $\kappa$-complexes, which includes $\check{\textrm{C}}$ech complexes as a special example.
We also remark that the papers \cite{YA} and \cite{YSA} discuss the limiting behaviour of Betti numbers of random $\check{\textrm{C}}$ech complexes built over stationary point processes.

The paper is organized as follows. We give
the statement of our results after introducing some notation and
fundamental facts in Section \ref{sec:Preliminaries and the results}.
Some examples of marked point processes and $\kappa$-filtered complexes
are also presented in this section.
In Section \ref{sec:proofs},
we show the law of large numbers for persistent Betti numbers
(Theorem \ref{th:LLN for PBNs}) to prove the main theorem (Theorem \ref{th:LLN for PDs}).

\section{Preliminaries and the results}\label{sec:Preliminaries and the results}

\subsection{$\kappa$-filtered complexes}\label{subsec:kappa}
For a topological space $S$,
let $\cF(S)$ be the collection of all finite non-empty subsets in $S$.
Given a function $f$ on $\cF(S)$, there exists a permutation
invariant function $f_{k}$ on $S^{k}$ such that
$f_{k}(s_{1}, s_{2}, \dots, s_{k})=f(\{s_{1}, s_{2}, \dots, s_{k}\})$
for any positive integer $k$.
We say a function $f$ on $\cF(S)$ is measurable
if the permutation invariant functions $\{f_{k}\}$ are Borel measurable.
In this paper, we extend the $\kappa$-filtration for (unmarked)
point processes introduced in \cite{HST} to that for marked ones.
Let $\M$ be a complete separable metric space, which stands
for the set of marks, and
$\kappa\,:\,\cF(\R^{d}\times \M)\to [0, \infty)$ a measurable function satisfying the following:

\begin{itemize}
  \item[(K1)]
        $\kappa(A)\le \kappa(B)$ if $A\subset B$.

  \item[(K2)]
        $\kappa$ is invariant under the translations
        acting on the first component only, i.e.,
        \[
          \kappa(T_{a}(A))=\kappa(A)
        \]
        for any $a\in \R^{d}$ and for any $A\in \cF(\R^{d}\times \M)$,
        where $T_{a}\,:(x, m)\mapsto (x+a, m)$.

  \item[(K3)]
        There exists an increasing function
        $\rho\,:\,[0, \infty)\to [0, \infty)$ such that
        \[
          \norm{x-y}\le \rho(\kappa(\{(x, m), (y, n)\}))
        \]
        for all $(x, m)$, $(y, n)\in \R^{d}\times \M$,
        where $\norm{\cdot}$ is the Euclidean norm on $\R^{d}$.
\end{itemize}
Let $\pi\,:\,\R^{d}\times \M\to \R^{d}$ be the projection with respect to
the first component. We say $\Xi\in
  \cF(\R^{d}\times \M)$ is a simple marked point set if
for any $x\in \R^{d}$, $\#(\Xi\cap \pi^{-1}\{x\})\le 1$ holds, where $\#A$ is the number of elememts in $A$.
For any simple marked point set $\Xi$, we write as $\Xi_{g}=\pi(\Xi)$. The projection $\pi$
naturally induces the bijection
\[
  \cF(\Xi)\ni \sigma\mapsto
  \sigma_{g}\in \cF(\Xi_{g}).
\]
Once a simple marked point set $\Xi$ is fixed,
each subset $\sigma=\{(x_{0}, m_{0}), (x_{1}, m_{1}), \dots, (x_{q}, m_{q})\}$ of $\Xi$ can be regarded as a finite point configuration $\sigma_{g}=\{x_{0}, x_{1},\dots, x_{q}\}$ in $\R^{d}$ with marks $\{m_{0}, m_{1},\dots, m_{q}\}$ in $\M$. By the definition of simple marked point sets we see
that each $x\in \Xi_{g}$ has a unique mark $m\in \M$
with $(x, m)\in \Xi$.

Given a simple marked point set $\Xi$, we construct a filtration
\begin{equation}
  \K(\Xi)=\{K(\Xi, t)\}_{t\ge 0}
\end{equation}
of simplicial complexes from the simple marked point  set
$\Xi$ and
the function $\kappa$ by
\[
  K(\Xi, t)=\{\sigma_{g}\subset \Xi_{g}\,:\,\kappa(\sigma)\le t\},
\]
i.e., $\kappa(\sigma)$ is the birth time of a
simplex $\sigma_{g}$ in the filtration $\K(\Xi)$.
Note that whether or not a $q$-simplex $\sigma_{g}=\{x_{0}, x_{1},\dots, x_{q}\}
  \subset \Xi_{g}$ belongs to $K(\Xi, t)$ depends not only on itself but also on the marked set
$\sigma=\{(x_{0}, m_{0}), (x_{1}, m_{1}), \dots, (x_{q}, m_{q})\}\subset \Xi$. We call
$\K(\Xi)=\{K(\Xi, t)\}_{t\ge 0}$ the $\kappa$-filtered complex built over $\Xi$.
We also note that the conditions (K1) and (K3) of $\kappa$
yield the following diameter bound
\[
  \diam\sigma_{g} \le \rho(t)
\]
for any simplex $\sigma\in K(\Xi, t)$. Indeed, for any
$\sigma \in K(\Xi, t)$ and any $x, y\in \sigma_{g}$ we take $m, n\in \M$
with $(x, m), (y, n)\in \sigma$, then it is easy to see that
\[
  \norm{x-y}\le \rho(\kappa(\{(x, m), (y, n)\}))
  \le \rho(\kappa(\sigma))\le \rho(t).
\]

\begin{example}[$\check{\textrm{C}}$ech and Vietoris-Rips filtered complex with various sizes]\label{ex.K1}
  For a fixed $R>0$, let $\M$ be the closed interval $[0, R]$. Fundamental examples of $\kappa$ on $\cF(\R^{d}\times \M)$ are
  \[
    \begin{aligned}
      \kappa_{C}(\sigma)
       & =\inf_{w\in \R^{d}}\max_{(x, r)\in \sigma}(\norm{x-w}-r)^{+}                                    \\	\text{and} \quad
      \kappa_{R}(\sigma)
       & =\max_{(x_{1}, r_{1}), (x_{2}, r_{2})\in \sigma}\frac{(\norm{x_{1}-x_{2}}-r_{1}-r_{2})^{+}}{2},
    \end{aligned}
  \]
  where $a^{+}=\max\{a, 0\}$ for $a\in \R$.
  It is easy to see that they satisfy (K1), (K2), and (K3) with $\rho(t)=2t+2R$.
  We denote the corresponding $\kappa$-filtered complexes
  built over a simple marked point set $\Xi$ by
  $\C(\Xi)=\{C(\Xi, t)\}_{t\ge 0}$ and $\R(\Xi)=\{R(\Xi, t)\}_{t\ge 0}$ respectively. For $\sigma\in \cF(\Xi)$, we see
  that
  \[
    \begin{aligned}
      \kappa_{C}(\sigma)\le t
       & \Leftrightarrow \bigcap_{(x, r)\in \sigma}\overline{B}_{t+r}(x)\neq\emptyset,                      \\
      \kappa_{R}(\sigma)\le t
       & \Leftrightarrow \overline{B}_{t+r_{1}}(x_{1})\cap \overline{B}_{t+r_{2}}(x_{2})\neq\emptyset \quad
      \text{ for any } (x_{1}, r_{1}), (x_{2}, r_{2})\in \sigma,
    \end{aligned}
  \]
  where $\overline{B}_{r}(x)=\{y\in \R^{d}\,:\, \norm{y-x}\le r\}$ is the closure of the
  open ball $B_{r}(x)$ of radius $r$ centered at $x$.
  Hence $C(\Xi, t)$ and
  $R(\Xi, t)$ are the so-called
  $\check{\textrm{C}}$ech complex and Vietoris-Rips complex of the family of balls $\{\overline{B}_{t+r}(x)\}_{(x, r)\in \Xi}$.
\end{example}

\begin{example}[$\check{\textrm{C}}$ech and Vietoris-Rips filtered complex with various growth speeds]\label{ex of  kappa 2}
  Let $\M$ be a finite family $\{\bm{r}_{i}(\cdot)\}_{i\in I}$ of
  right continuous, strictly increasing functions on $[0, \infty)$.
  We define functions on $\cF(\R^{d}\times \M)$ by
  \[
    \begin{aligned}
      \kappa_{C}(\sigma)
       & =\inf_{w\in \R^{d}}\max_{(x, \bm{r})\in \sigma}\bm{r}^{-1}(\norm{x-w})                                      \\
      \text{ and } \quad \kappa_{R}(\sigma)
       & =\max_{(x_{1}, \bm{r}_{1}), (x_{2}, \bm{r}_{2})\in \sigma}(\bm{r}_{1}+\bm{r}_{2})^{-1}(\norm{x_{1}-x_{2}}),
    \end{aligned}
  \]
  where $\bm{r}^{-1}(t)=\inf\{s\ge 0\,:\,\bm{r}(s)\ge t\}$.
  One can show that
  \[
    \begin{aligned}
      \kappa_{C}(\sigma)\le t
       & \Leftrightarrow \bigcap_{(x, \bm{r})\in \sigma}\overline{B}_{\bm{r}(t)}(x)\neq\emptyset,                                                                                     \\
      \kappa_{R}(\sigma)\le t
       & \Leftrightarrow \overline{B}_{\bm{r}_{1}(t)}(x_{1})\cap \overline{B}_{\bm{r}_{2}(t)}(x_{2})\neq\emptyset 	\text{ for any } (x_{1}, \bm{r}_{1}), (x_{2}, \bm{r}_{2})\in \sigma
    \end{aligned}
  \]
  in the same way as Example \ref{ex.K1} above.
  In this case, (K3) is satisfied with $\rho(t)=2\max_{i\in I}\bm{r}_{i}(t)$.
  The corresponding $\kappa$-filtered complexes are
  the $\check{\textrm{C}}$ech complexes and Vietoris-Rips complexes of the family of balls
  $\{\overline{B}_{\bm{r}(t)}(x)\}_{(x, \bm{r})\in \Xi}$ for a simple marked point set $\Xi\in \cF(\R^{d}\times \M)$.
\end{example}

\begin{example}[$\check{\textrm{C}}$ech filtered complex with various shapes]\label{ex of kappa 3}
  Let $\M$ be a finite family $\{C_{i}\}_{i\in I}$ of  bounded convex sets in $\R^{d}$ satisfying that
  $0\in \inte C_{i}$ for every $i\in I$, where $\inte C$ is the interior of $C$. We put $f_{C}(z)=\inf\{s\ge 0\,:\,z\in sC\}$ for a convex set $C$ and $z\in \R^{d}$. Consider the function on $\cF(\R^{d}\times\M)$ defined by
  \[
    \kappa(\sigma)=\inf_{w\in \R^{d}}
    \max_{(x, C)\in \sigma}f_{C}(w-x).
  \]
  This satisfies (K3) with $\rho(t)=2t\max_{i\in I}\diam C_{i}$. For any simple marked point set $\Xi\in \cF(\R^{d}\times \M)$, it is easy to see that the corresponding
  $K(\Xi, t)$ is the $\check{\textrm{C}}$ech  complex of the family of sets $\{t\overline{C}+x\}_{(x, C)\in \Xi}$.
\end{example}

\subsection{Persistent homologies and persistence diagrams}\label{subsec:PH and PD}
In what follows, we fix a function $\kappa$ satisfying
the conditions (K1)--(K3) in Section \ref{subsec:kappa}.
Now we give a brief introduction of persistent homology, persistence diagrams, and
persistent Betti numbers for the $\kappa$-filtered complex $\K(\Xi)$.
Let $\F$ be a field. Given a nonnegative integer $q$ and $t\ge 0$, we denote by
$H_{q}(K(\Xi, t))$ the $q$th homology group of
the simplicial complex $K(\Xi, t)$ with coefficients in $\F$.
For $r\le s$, the inclusion
$K(\Xi, r)\hookrightarrow K(\Xi, s)$ induces the linear map
$\iota_{r}^{s}\,:\,H_{q}(K(\Xi, r))\to H_{q}(K(\Xi, s))$.
We put
$H_{q}(\K(\Xi))=(H_{q}(\{\K(\Xi, t)\}_{t\ge 0},
  \{\iota_{r}^{s}\}_{s\ge r\ge 0})$ and call it the $q$th persistent homology
(or persistence module) of $K(\Xi)$.
It is well-known that there exist a unique nonnegative
integer $n_{q}$ and $b_{i}, d_{i}\in [0, \infty]$ with $b_{i}<d_{i}$,
$i=1,2, \dots, n_{q}$, such that the $q$th persistent homology
$H_{q}(\K(\Xi))$ has a decomposition property
\begin{equation}\label{eq:decomposition of PD}
  H_{q}(\K(\Xi))\simeq \bigoplus_{i=1}^{n_{q}}I(b_{i}, d_{i}),
\end{equation}
where $I(b_{i}, d_{i})=(U_{r}, f_{r}^{s})$ consists of a family of vector spaces
\[
  U_{r}
  =\begin{cases}
    \F \quad & b_{i}\le r<d_{i},   \\
    0 \quad  & \text{ otherwise, }
  \end{cases}
\]
and the identity map $f_{r}^{s}=\id_{\F}$ for $b_{i}\le r\le s<d_{i}$. Each
$I(b_{i}, d_{i})$ in (\ref{eq:decomposition of PD}) describes that a topological
feature ($q$th homology class) appears at $t=b_{i}$,
persists for $b_{i}\le t<d_{i}$, and disappears at $t=d_{i}$ in
$\K(\Xi)$. We call the pair $(b_{i}, d_{i})$ its birth-death pair.
The $q$th persistence diagram of $\K(\Xi)$ is defined by a multiset
\[
  D_{q}(\K(\Xi))
  =\{(b_{i}, d_{i})\in \Delta\,:\,i=1,2, \dots, n_{q}\},
\]
where $\Delta=\{(x, y)\in [0, \infty]\times[0, \infty]\,:\,x<y\}$.
Let $m_{b, d}$ be the multiplicity of the point $(b, d)\in D_{q}(\Xi)$ and $\xi_{q}(\K(\Xi))$ the counting measure on $\Delta$ given by
\[
  \xi_{q}(\K(\Xi))=\sum_{(b, d)}m_{b, d}\delta_{(b, d)},
\]
where $\delta_{(x, y)}$ is the Dirac measure at $(x, y)\in \Delta$.
We identify the persistence diagram $D_{q}(\K(\Xi))$ with
the counting measure $\xi_{q}(\K(\Xi))$.
The $q$th $(r, s)$-persistent Betti number is also defined by
\[
  \beta_{q}^{r, s}(\K(\Xi))
  =\dim \frac{Z_{q}(K(\Xi, r))}
  {Z_{q}(K(\Xi, r))\cap B_{q}(K(\Xi, s))},
\]
where $Z_{q}(K(\Xi, r))$ and
$B_{q}(K(\Xi, r))$ are the $q$th cycle group and boundary
group of $K(\Xi, r)$, respectively.
It is easy to see that this number is equal to the rank of $\iota_{r}^{s}\,:\,H_{q}(K(\Xi, r))\to H_{q}(K(\Xi, s))$. By definition of the persistent Betti number, we have
\begin{equation}\label{eq:PBN}
  \beta_{q}^{r, s}(\K(\Xi))=\sum_{b\le r, s<d}m_{b, d}
  =\xi_{q}(\K(\Xi))([0, r]\times (s, \infty]).
\end{equation}
Therefore the persistence Betti number $\beta_{q}^{r, s}$ counts the
number of birth-death pairs in the persistence diagram
$D_{q}(\K(\Xi))$ located in the gray region of Figure
\ref{fig:PBN}.
\begin{figure}[h]
  \includegraphics[width=5cm]{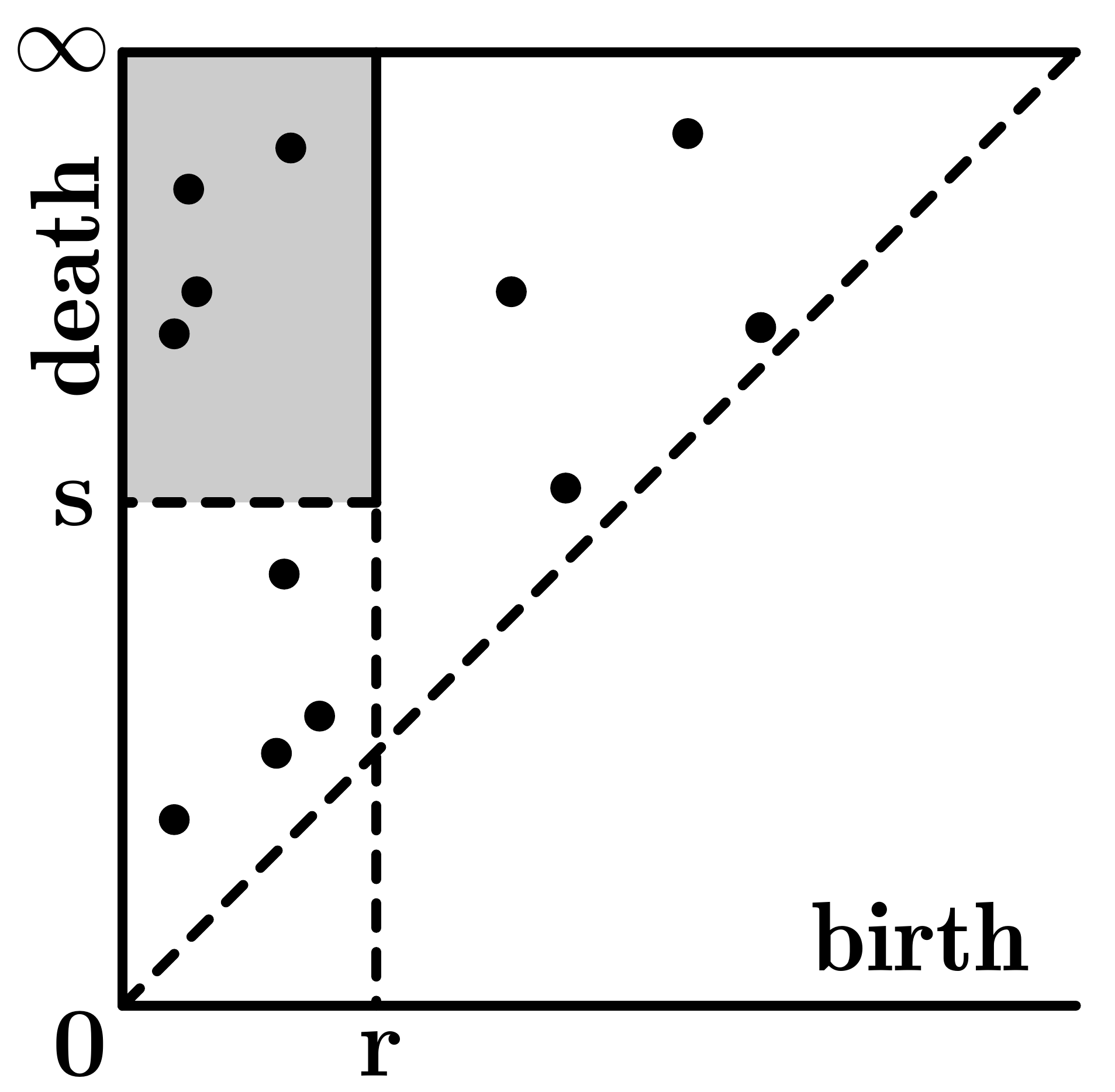}
  \caption{$\beta_{q}^{r, s}$ counts the number of birth-death pairs
    in the gray region.}
  \label{fig:PBN}
\end{figure}
Details for these facts can be found in \cite{ELZ}, \cite{HST} and \cite{ZC} for example.

\subsection{Marked point processes}
Now we consider marked point processes.
Let $X$ be a complete separable metric space and
$\cB(X)$ the Borel $\sigma$-field on $X$.
A Borel measure $\mu$ on $X$
is boundedly finite if $\mu(A)<\infty$ for every bounded Borel set $A$.
We say that a sequence $\{\mu_{n}\}$ of boundedly finite measures on $X$
converges to a boundedly finite measure $\mu$ on $X$ in the
$w^{\#}$-topology if
\begin{equation}\label{def:weak hash}
  \int_{X}f\,d\mu_{n}\to \int_{X}f\,d\mu \text{ as } n\to \infty
\end{equation}
for all bounded continuous functions $f$ on $X$ vanishing outside
a bounded set.
We denote by $\cM_{X}^{\#}$ the totality of boundedly finite
measures on $\cB(X)$. $\cM_{X}^{\#}$ is a
complete separable metric space under the $w^{\#}$-topology.
The corresponding $\sigma$-field $\cB(\cM_{X}^{\#})$ coincides with
the smallest $\sigma$-field with respect to which the mappings
$\mu\mapsto \mu(A)$ are measurable for all $A\in \cB(X)$.
If $X$ is a locally compact Hausdorff space with countable base,
we can take a metric so that $X$ is complete and every bounded subset
of $X$ is relatively compact. Then a Borel measure is boundedly finite
if and only if it is a Radon measure and $w^{\#}$-convergence coincides
with vague convergence. We recall that a Radon measure is a measure on $X$
taking finite values on compact sets and a sequence
$\{\mu_{n}\}$ of Radon measures on $X$
converges to a Radon measure $\mu$ on $X$ vaguely
(or in the vague topology) if $(\ref{def:weak hash})$ holds
for each continuous function $f$ on $X$ vanishing outside a compact set.
In this case, we write $\mu_{n}\xrightarrow{v} \mu$.
Let $\cN_{X}^{\#}$ be the totality of boundedly finite integer-valued
measures. We call a measure in $\cN_{X}^{\#}$ a counting measure for short. For a counting measure $\mu$ on $X$, there exist sequences of
positive integers $\{k_{i}\}$ and points $\{x_{i}\}$ in $X$ with
at most finitely many $x_{i}$ in any bounded Borel set such that
\[
  \mu=\sum_{i}k_{i}\delta_{x_{i}}.
\]
Note that $\cN_{X}^{\#}$ is a closed subset of $\cM_{X}^{\#}$.

Let $(\Omega, \cF, \P)$ be a probability space. An
$(\cM_{X}^{\#}, \cB(\cM_{X}^{\#}))$ (resp., $(\cN_{X}^{\#}, \cB(\cN_{X}^{\#}))$-valued random variable $\xi$ on $(\Omega, \cF, \P)$
is called a random measure (resp., point process) on $X$.
A point process $\xi$ is typically identified with the random point configuration of its atoms.
The expectation measure (or mean measure) of $\xi$ is defined so that
$M(A)=\E[\xi(A)]$ for any $A\in \cB(X)$. It is often denoted by $\E[\xi]$.
We say that a point process $\xi$ is simple if
\[
  \P(\xi(\{x\})=0 \text{ or } 1 \text{ for any } x\in X)=1.
\]
A marked point process on $\R^{d}$ with marks in $\M$ is a point process
$\Phi$ on $\R^{d}\times \M$ whose marginal point process
$\Phi_{g}(\cdot)=\Phi(\cdot\times\M)$ on $\R^{d}$ is a simple
point process on $\R^{d}$.
The point process $\Phi_{g}$ is called the ground process of $\Phi$.
We say that the ground process $\Phi_{g}$ has all finite moments if
$\E[\Phi_{g}(A)^{p}]<+\infty$ for every bounded $A\in \cB(\R^{d})$ and every
$p\ge 1$.
The translations $\{T_{a}\}_{a\in \R^{d}}$ on $\R^{d}\times \M$ induce
the translations $\{{T_{a}}_{\ast}\}_{a\in \R^{d}}$ on
$\cN_{\R^{d}\times\M}^{\#}$ defined by
\[
  ({T_{a}}_{\ast}\mu)(A)=\mu(T_{a}^{-1}A)
\]
for $a\in \R^{d}$ and $A\in \cB(\R^{d}\times\M)$. A marked point process
is called stationary if its probability distribution
is translation invariant. A stationary marked point process
$\Phi$ is called ergodic if the only members $B$ of
$\cB(\cN_{\R^{d}\times\M}^{\#})$ with
$\P\circ\Phi^{-1}({T_{a}}_{\ast}B\Delta B)=0$ for all $a\in \R^{d}$
satisfy $\P\circ\Phi^{-1}(B)=0 \text{ or } 1$.

\begin{example}[point process with i.i.d.\ marks]
  Let $\Phi_{g}$ be a point process on $\R^{d}$ and
  $\{X_{i}\}$ a measurable enumeration of $\Phi_{g}$, that
  is, $\{X_{i}\}$ is a sequence of $\R^{d}$-valued random variables so that $\Phi_{g}=\sum_{i}\delta_{X_{i}}$ a.s.
  We take an i.i.d.\ sequence of $\M$-valued random variables $\{M_{i}\}$ such that $\Phi=\{X_{i}\}$
  and $\{M_{i}\}$ are independent. A marked point process
  on $\R^{d}$ with marks in $\M$ is defined by
  \[
    \Phi=\sum_{i}\delta_{(X_{i}, M_{i})}.
  \]
  If the point process $\Phi$ is stationary (and ergodic), then so is $\Phi$.
\end{example}

\begin{example}
  Let $\Phi_{g}$ be a simple stationary (ergodic) point process on $\R^{d}$
  and $\{X_{i}\}$ a measurable enumeration of $\Phi_{g}$.
  For a fixed $R>0$ and for each $i$, we define a $\{0, 1\}$-valued random variable $M_{i}$ by
  \begin{equation}\label{def.M_{i}}
    M_{i}
    =\begin{cases}
      1, & \text{ if there exists } j\neq i \text{ such that } |X_{i}-X_{j}|\le R, \\
      0, & \text{ otherwise. }
    \end{cases}
  \end{equation}
  The point process on $\R^{d}\times \{0, 1\}$ defined by
  \[
    \Phi=\sum_{i}\delta_{(X_{i}, M_{i})}
  \]
  is a marked point process.
  In general, for measurable maps $M_{i}\colon \cN_{\R^{d}}^{\#}\to \M\,(i\ge 1)$, the point process on $\R^{d}\times \M$ defined by
  \[
    \Phi=\sum_{i}\delta_{(X_{i}, M_{i}(\Phi_{g}))}
  \]
  is a stationary marked point process.

  Incidentally, for $M_{i}\,(i\ge 1)$ in \eqref{def.M_{i}},  the point process
  \[
    \Phi_{I}=\sum_{i: M_{i}=0}\delta_{X_{i}}
  \]
  on $\R^{d}$ is called a Mat\'ern type I construction of $\Phi$ (See \cite{NB}).
\end{example}

Other examples and basic facts for marked point processes are available in \cite{DV1}
and \cite{DV2}.

\subsection{Main theorems}
In order to state the main results we introduce the notion of convex averaging nets in $\R^{d}$.
Let $(\cN, \le)$ be a linearly ordered set.
A family $\cA=\{A_{n}\}_{n\in\cN}$ of bounded
Borel sets in $\R^{d}$ is a convex averaging net if
\begin{itemize}
  \item[(i)]
        $A_{n}$ is convex for each $n\in \cN$,

  \item[(ii)]
        $A_{n}\subset A_{m}$ for $n\le m$, and

  \item[(iii)]
        $\displaystyle{\sup_{n\in \cN}r(A_{n})=\infty}$, where
        $r(A)=\sup\{r>0\,:\,A \text{ contains a ball of radius } r\}$.
\end{itemize}
Given a marked point process $\Phi$ and
$A\in \cB(\R^{d})$,
we denote the restricted marked point process $\Phi$ on $A\times \M$ by $\Phi_{A}$, i.e., $\Phi_{A}(\cdot)
  =\Phi(\cdot\cap (A\times \M))$.
Note that $\Phi_{A}$ can be regarded as a random simple marked point set
for any bounded $A$.
For any convex averaging net $\cA=\{A_{n}\}_{n\in \cN}$,
a random $\kappa$-filtered complex
with parameter $n\in \cN$ is defined by
$\K(\Phi_{A_{n}})
  =\{K(\Phi_{A_{n}}, t)\}_{t\ge 0}$. For the sake of simplicity, we
often denote the corresponding $q$th persistence diagram
$\xi_{q}(\K(\Phi_{A_{n}}))$
and $q$th $(r, s)$-persistent Betti number
$\beta_{q}^{r, s}(\K(\Phi_{A_{n}}))$ by
$\xi_{q, A_{n}}$ and $\beta_{q, A_{n}}^{r, s}$
respectively.

Now we are in a position to state the main theorem.

\begin{theorem}\label{th:LLN for PDs}
  Let $\Phi$ be a stationary marked point process and suppose
  its ground process $\Phi_{g}$ has all finite moments. Then for any nonnegative  integer $q$,
  there exists a Radon measure $\nu_{q}$ on $\Delta$
  such that for any convex averaging net $\cA=\{A_{n}\}_{n\in\cN}$ in $\R^{d}$,
  \[
    \frac{1}{\ell(A_{n})}\E[\xi_{q, A_{n}}]\xrightarrow{v} \nu_{q}
    \quad \text{as} \quad n\to \infty,
  \]
  where $\ell$ is the $d$-dimensional Lebesgue measure.
  Furthermore if $\Phi$ is ergodic, then
  \[
    \frac{1}{\ell(A_{n})}\xi_{q, A_{n}}\xrightarrow{v} \nu_{q}
    \quad \text{as} \quad n\to \infty \text{ a.s. }
  \]
\end{theorem}

Theorem \ref{th:LLN for PDs} can be proved by a general theory of Radon measures
for the vague convergence and the following law of
large numbers for persistent Betti numbers.

\begin{theorem}\label{th:LLN for PBNs}
  Let $\Phi$ be a stationary marked point process and suppose
  its ground process $\Phi_{g}$ has all finite moments. Then, for
  any $0\le r\le s<\infty$ and nonnegative integer $q$, there exists
  a nonnegative number $\hat{\beta}_{q}^{r, s}$
  such that for any convex averaging net $\cA=\{A_{n}\}_{n\in \cN}$ in $\R^{d}$,
  \[
    \frac{1}{\ell(A_{n})}
    \E[\beta_{q, A_{n}}^{r, s}]
    \to \hat{\beta}_{q}^{r, s}
    \qquad \text{as} \quad n\to \infty.
  \]
  Furthermore, if $\Phi$ is ergodic, then
  \[
    \frac{1}{\ell(A_{n})}\beta_{q,A_{n}}^{r, s}
    \to \hat{\beta}_{q}^{r, s}
    \quad \text{ as } \quad n\to \infty \text{ a.s. }
  \]
\end{theorem}
The proofs of Theorem \ref{th:LLN for PDs} and Theorem \ref{th:LLN for PBNs}
will be given in the next Section.
They are shown in the same way as \cite[Theorem $1.5$ and Theorem $1.11$]{HST}, in which such limit theorems for stationary (unmarked) point processes were proved.

\section{Proof of Theorems \ref{th:LLN for PDs} and \ref{th:LLN for PBNs}}\label{sec:proofs}

The aim of this section is to prove Theorem \ref{th:LLN for PDs} and Theorem
\ref{th:LLN for PBNs}.

\subsection{Convergence of persistent Betti numbers}\label{sec:convergence of PBNs}
Let $M, h$ be positive numbers and $A\in \cB(\R^{d})$. We put
\[
  \begin{aligned}
    \Lambda_{M}                 & =[-M/2, M/2)^{d},                           \\
    \underline{A}^{(M)}
                                & =\bigsqcup\{\Lambda_{M}+z\,:\,z\in M\Z^{d}
    \text{ and } (\Lambda_{M}+z)\subset A\},                                  \\
    \overline{A}^{(M)}
                                & =\bigsqcup\{\Lambda_{M}+z\,:\,z\in M\Z^{d}
    \text{ and } (\Lambda_{M}+z)\cap A\neq\emptyset\},                        \\
    \text{ and } \partial A^{h} & =\{x\in \R^{d}\,:\,d(x, \partial A)\le h\},
  \end{aligned}
\]
where
$M\Z^{d}=\{Mz\,:\,z\in \Z^{d}\}$ and
$\displaystyle{d(x, \partial A)=\inf_{y\in \partial A}|x-y|}$.
Fundamental results treated in this paper for
convex averaging nets are summarized in the next proposition.
\begin{proposition}\label{prop:lim for convex sets}
  Let $\cA=\{A_{n}\}_{n\in \cN}$ be a convex averaging
  net in $\R^{d}$. Then for any $M>0$ and $h>0$, as $n\to \infty$,
  \begin{equation}\label{prop:lim for convex sets1}
    \frac{\ell(\underA_n^{(M)})}{\ell(A_{n})}\to 1,
    \quad \frac{\ell(\overA_n^{(M)})
      \setminus\underA_{n}^{(M)})}{\ell(A_{n})}\to 0,
  \end{equation}
  and
  \begin{equation}\label{prop:lim for convex sets2}
    \frac{\ell(\partial A_{n}^{h})}{\ell(A_{n})}\to 0.
  \end{equation}
\end{proposition}

Proposition \ref{prop:lim for convex sets} is a special case of \cite[Lemma 3.1]{NZ1}.
For (\ref{prop:lim for convex sets1}) and (\ref{prop:lim for convex sets2}),
see \cite[Lemma 1]{F} and \cite[Lemma 2]{NZ2},
respectively.

Next we need a version of the multi-dimensional
ergodic theorem for stationary ergodic marked point processes.

\begin{proposition}\label{prop:ergodic th}
  Let $\Phi$ be a stationary ergodic marked point process and $Z\in L^{p}(\P\circ \Phi^{-1})$ for $1\le p<+\infty$. If $\cA=\{A_{n}\}_{n\in \cN}$
  is a convex averaging net, then for each $M>0$
  \[
    \begin{aligned}
      \lim_{n\in \cN}\frac{1}{\ell(A_{n})}\sum_{z\in M\Z^{d}\cap A_{n}}Z(T_{-z\ast}\Phi)
       & =\frac{1}{M^{d}}\E[Z(\Phi)] \quad \text{ a.s. } \\
      \text{ and } \quad
      \lim_{n\in \cN}\frac{1}{\ell(A_{n})}\sum_{z\in M\Z^{d}\cap \underA_n^{(M)}}Z(T_{-z\ast}\Phi)
       & =\frac{1}{M^{d}}\E[Z(\Phi)] \quad \text{ a.s. }
    \end{aligned}
  \]
\end{proposition}

\begin{proof}
  Take any $M>0$. Applying \cite[Theorem 3.7 and Corollary 3.10]{NZ1} to the
  probability space $(\cN_{X}^{\#}, \cB(\cN_{X}^{\#}), \P\circ\Phi^{-1})$ and the translations $\{T_{z\ast}\}_{z\in M\Z^{d}}$, we have
  \[
    \begin{aligned}
      \lim_{n\in \cN}\frac{1}{\ell(A_{n})}\sum_{z\in M\Z^{d}\cap A_{n}}Z(T_{-z\ast}\Phi)
       & =\frac{1}{M^{d}}\E[Z(\Phi)|\Phi^{-1}\cI] \quad \text{ a.s. }  \\
      \text{ and } \quad
      \lim_{n\in \cN}\frac{1}{\ell(A_{n})}\sum_{z\in M\Z^{d}\cap \underA_n^{(M)}}Z(T_{-z\ast}\Phi)
       & =\frac{1}{M^{d}}\E[Z(\Phi)|\Phi^{-1}\cI] \quad \text{ a.s., }
    \end{aligned}
  \]
  where $\cI$ is the invariant $\sigma$-field in $\cN_{\R^{d}
      \times\M}^{\#}$ under the translations $\{T_{z\ast}\}_{z\in M\Z^{d}}$.
  Rotating $M\Z^{d}$ if necessary,
  we may assume $\cI$ is trivial, that is, for every
  $I\in \cI$, $\P\circ\Phi^{-1}(I)=0$ or $1$ (see \cite[Theorem 1]{PS}). Therefore we have
  $\E[Z(\Phi)|\Phi^{-1}\cI]=\E[Z(\Phi)]$ a.s.
\end{proof}

Let $S_{q}(\Xi, t)$ be the number of
$q$-simplices in $\K(\Xi, t)$ for a
simple marked point set $\Xi$.
The following limit theorems for $S_{q}(\Phi, t)$ play important roles
in the proof of Theorem \ref{th:LLN for PBNs}.

\begin{lemma}\label{lem:conv of Sq}
  Let $\Phi$ be a stationary ergodic marked point process and suppose
  its ground process $\Phi_{g}$ has all finite moments.
  Then for any nonnegative integer $q$,
  $t\ge 0$, $M>0$, and convex averaging net $\cA=\{A_{n}\}_{n\in \cN}$,
  \[
    \lim_{n\in \cN}\frac{1}{\ell(A_{n})}S_{q}(\Phi_{\underA_n^{(M)}}, t)
    =\lim_{n\in
      \cN}\frac{1}{\ell(A_{n})}S_{q}(\Phi_{\overA_n^{(M)}}, t)
    =\lim_{n\in \cN}\frac{1}{\ell(A_{n})}S_{q}(\Phi_{A_{n}}, t) \text{ a.s. }
  \]
\end{lemma}

\begin{proof}
  The proof is similar to that of \cite[Lemma 3.2]{YSA}.
  Consider the function defined by
  \[
    h_{q, t}^{(M)}(\Phi)
    =\frac{1}{q+1}\sum_{x\in \Phi_{g}\cap\Lambda_{M}}
    \#\{q \text{-simplices in } \K(\Phi, t)
    \text{ containing } x\}.
  \]
  We recall that $\diam \sigma_{g}\le \rho(t)$ for every $\sigma_{g}\in \K(\Phi_{A_{n}}, t)$. Hence we obtain
  \[
    \begin{aligned}
      \sum_{z\in M\Z^{d}\cap
        (A_{n}\setminus \partial A_{n}^{\rho(t)+2\sqrt{d}M})}h_{q, t}^{(M)}(T_{-z \ast}\Phi)
       & \le S_{q}(\Phi_{\underA_n^{(M)}}, t)
      \le S_{q}(\Phi_{A_{n}}, t)              \\
       & \le S_{q}(\Phi_{\overA_n^{(M)}}, t)
      \le \sum_{z\in M\Z^{d}\cap
        (A_{n}\cup \partial A_{n}^{\sqrt{d}M})}h_{q, t}^{(M)}(T_{-z \ast}\Phi).
    \end{aligned}
  \]
  Since the ground process $\Phi_{g}$ of $\Phi$ has all finite moment,
  we have $\E[h_{q, t}^{(M)}(\Phi)]\le
    \E[\Phi_{g}(\Lambda_{M}\cup \partial \Lambda_{M}^{\rho(t)})^{q+1}]<+\infty$. If we notice the fact that
  $\{A_{n}\setminus \partial A_{n}^{\rho(t)+2\sqrt{d}M}\}_{n\in \cN}$ is also
  a convex averaging net, we see from Proposition \ref{prop:lim for convex sets} and Proposition \ref{prop:ergodic th} that
  \[
    \begin{aligned}
       &
      \hspace{12pt}\frac{1}{\ell(A_{n})}\sum_{z\in
      M\Z^{d}\cap (A_{n}\setminus \partial A_{n}^{\rho(t)+2\sqrt{d}M})}h_{q, t}^{(M)}(T_{-z \ast}\Phi)                                                 \\
       & =\frac{\ell(A_{n}\setminus
        \partial A_{n}^{\rho(t)+2\sqrt{d}M})}{\ell(A_{n})}
      \cdot\frac{1}{\ell(A_{n}\setminus
      \partial A_{n}^{\rho(t)+2\sqrt{d}M})}\sum_{z\in M\Z^{d}\cap (A_{n}\setminus \partial A_{n}^{\rho(t)+2\sqrt{d}M})}h_{q, t}^{(M)}(T_{-z \ast}\Phi) \\
       & \to \frac{1}{M^{d}}\E[h_{q, t}^{(M)}(\Phi)]
      \text{ as } n \to \infty \text{ a.s. }
    \end{aligned}
  \]
  We can similarly show that
  \[
    \frac{1}{\ell(A_{n})}\sum_{z\in M\Z^{d}\cap
      (A_{n}\cup \partial A_{n}^{\sqrt{d}M})}h_{q, t}^{(M)}(T_{-z \ast}\Phi)\to \frac{1}{M^{d}}\E[h_{q, t}^{(M)}(\Phi)]
    \text{ as } n\to \infty \text{ a.s. }
  \]
  Therefore we reach the desired result.
\end{proof}

Now we state a basic estimate on the persistent
Betti numbers for nested filtered complexes $\K^{(1)}\subset \K^{(2)}$.
The proof of the following lemma is given in \cite[Lemma 2.11]{HST}.

\begin{lemma}\label{lem:estimates for PBNs}
  Let $\K^{(1)}=\{K^{(1)}_{t}\}_{t\ge 0}$ and
  $\K^{(2)}=\{K^{(2)}_{t}\}_{t\ge 0}$ be filtered complexes with $K^{(1)}_{t}\subset K^{(2)}_{t}$ for $t\ge 0$. Then
  \[
    |\beta_{q}^{r, s}(\K^{(1)})-\beta_{q}^{r, s}(\K^{(2)})|
    \le \sum_{j=q, q+1}\#K^{(2)}_{s, j}\setminus K^{(1)}_{s, j}
    +\#\{\sigma\in K^{(1)}_{s, j}\setminus K^{(1)}_{r, j}
    \,:\,t^{(2)}_{\sigma}\le r\},
  \]
  where $K^{(i)}_{s, j}$ is the set of $j$-simplices in $K^{(i)}_{s}$, and $t^{(i)}_{\sigma}$ is the birth time
  of $\sigma$ in $\K^{(i)}$, $i=1,2$. In particular,
  if $t^{(1)}_{\sigma}=t^{(2)}_{\sigma}$ holds for any simplex $\sigma$ in $\K^{(1)}$,
  then
  \[
    |\beta_{q}^{r, s}(\K^{(1)})-\beta_{q}^{r, s}(\K^{(2)})|
    \le \sum_{j=q, q+1}\#K^{(2)}_{s, j}\setminus K^{(1)}_{s, j}.
  \]
\end{lemma}

Now we give the proof of Theorem \ref{th:LLN for PBNs}.

\textit{Proof of Theorem \ref{th:LLN for PBNs}}.
We first note that it can be proved that there exist $C_{q, t}\ge 0$ and
$\hat{\beta}_{q}^{r, s}\ge 0$ such that
\begin{equation}\label{inq:sq}
  \E[S_{q}(\Phi_{A}, t)]\le C_{q, t}\ell(A) \text{ for bounded } A\in \cB(\R^{d})
\end{equation}
and
\begin{equation}
  \lim_{M\to \infty}\frac{1}{M^{d}}
  \E[\beta_{q, \Lambda_{M}}^{r, s}]=\hat{\beta}_{q}^{r, s}
\end{equation}
in the same way as \cite[Theorem 1.11]{HST}.
Take $0\le r\le s<\infty$ and a nonnegative integer $q$ and fix them.
The set $\underA_n^{(M)}$ is decomposed into rectangles
\[
  \underA_n^{(M)}
  =\bigsqcup_{z\in M\Z^{d}\cap\underA_n^{(M)}}(\Lambda_{M}+z).
\]
We define a new filtered complex $\K^{\circ}(\Phi_{\underA_n^{(M)}})$ by
\[
  \K^{\circ}(\Phi_{\underA_n^{(M)}})
  =\bigsqcup_{z\in M\Z^{d}\cap\underA_n^{(M)}}
  \K(\Phi_{\Lambda_{M}+z}).
\]
From the second assertion in Lemma \ref{lem:estimates for PBNs}, we have
\begin{equation}\label{inq:PBN}
  |\beta_{q, A_{n}}^{r, s}-\beta_{q}^{r, s}(\K^{\circ}(\Phi_{\underA_n^{(M)}}))|
  \le \sum_{j=q, q+1}\left(\sum_{z\in
    M\Z^{d}\cap\underA_n^{(M)}}S_{j}(\Phi_{\partial
    \Lambda_{M}^{2\rho(s)}+z}, s)
  +S_{j}(\Phi_{A_{n}\setminus\underA_n^{(M)}}, s)\right).
\end{equation}
Since $\Phi$ is stationary and it is easy to see that
$\ell(\underA_n^{(M)})=\#\{M\Z^{d}\cap\underA_n^{(M)}\}\cdot M^{d}$, we have
\begin{equation}\label{est:1}
  \begin{aligned}
    \E[\beta_{q}^{r, s}(\K^{\circ}(\Phi_{\underA_n^{(M)}}))]
     & =\E\left[\sum_{z\in M\Z^{d}\cap\underA_n^{(M)}}\beta_{q, \Lambda_{M}+z}^{r, s}\right]
    =\#\{M\Z^{d}\cap\underA_n^{(M)}\}
    \cdot \E[\beta_{q, \Lambda_{M}}^{r, s}]                                                  \\
     & =\ell(\underA_n^{(M)})\cdot\frac{1}{M^{d}}
    \E[\beta_{q, \Lambda_{M}}^{r, s}].
  \end{aligned}
\end{equation}
In addition, we have
\begin{equation}\label{est:2}
  \begin{aligned}
    \E\left[\sum_{z\in
      M\Z^{d}\cap\underA_n^{(M)}}S_{j}(\Phi_{\partial \Lambda_{M}^{2\rho(s)}+z}, s)\right]
     & =\#\{M\Z^{d}\cap\underA_n^{(M)}\}
    \cdot \E[S_{j}(\Phi_{\partial \Lambda_{M}^{2\rho(s)}}, s)] \\
     & \le \ell(\underA_n^{(M)})
    \cdot \frac{C_{j, s}\ell(\partial \Lambda_{M}^{2\rho(s)})}{M^{d}}
  \end{aligned}
\end{equation}
and
\begin{equation}\label{est:3}
  \E[S_{j}(\Phi_{A_{n}\setminus\underA_n^{(M)}}, s)]\le C_{j, s}\ell(A_{n}\setminus\underA_n^{(M)})
\end{equation}
from (\ref{inq:sq}). Take $\varepsilon>0$. We can find $M>0$ such that
\begin{equation}\label{inq:M}
  \left|\frac{1}{M^{d}}
  \E[\beta_{q, \Lambda_{M}}^{r, s}]-\hat{\beta}_{q}^{r, s}\right|, \quad
  \sum_{j=q, q+1}\frac{C_{j ,s}\ell(\partial \Lambda_{M}^{2\rho(s)})}{M^{d}}<\varepsilon.
\end{equation}
By taking expectation on both sides of the inequality (\ref{inq:PBN}),
we see that the estimates (\ref{est:1}), (\ref{est:2}), and (\ref{est:3}) yield that
\[
  \begin{aligned}
     & \left|\frac{1}{\ell(A_{n})}\E[\beta_{q, A_{n}}^{r, s}]
    -\hat{\beta}_{q}^{r, s}\right|                                                         \\
     & \ \le\frac{1}{\ell(A_{n})}\E\left[\left|\beta_{q, A_{n}}^{r, s}
    -\beta_{q}^{r, s}(\K^{\circ}(\Phi_{\underA_n^{(M)}}))\right|\right]
    +\frac{\ell(\underA_n^{(M)})}{\ell(A_{n})}\left|\frac{1}{M^{d}}\E[\beta_{q, A_{n}}^{r, s}]
    -\hat{\beta}_{q}^{r, s}\right|                                                         \\
     & \ \ +\hat{\beta}_{q}^{r, s}\left|\frac{\ell(\underA_n^{(M)})}{\ell(A_{n})}-1\right| \\
     & \ \le \frac{\ell(\underA_n^{(M)})}{\ell(A_{n})}\varepsilon
    +\sum_{j=q, q+1}C_{j, s}\frac{\ell(A_{n}\setminus\underA_n^{(M)})}{\ell(A_{n})}+\frac{\ell(\underA_n^{(M)})}{\ell(A_{n})}\varepsilon+\hat{\beta}_{q}^{r, s}\left|\frac{\ell(\underA_n^{(M)})}{\ell(A_{n})}-1\right|.
  \end{aligned}
\]
Therefore we conclude that
\[
  \limsup_{n\in \cN}\left|\frac{1}{\ell(A_{n})}\E[\beta_{q, A_{n}}^{r, s}]
  -\hat{\beta}_{q}^{r, s}\right|\le 2\varepsilon.
\]
This implies the first assertion.

In order to prove the second assertion, we assume that $\Phi$ is
ergodic. By virtue of the multi-dimensional ergodic theorem mentioned in Proposition \ref{prop:ergodic th}, we see that
\begin{equation}\label{conv1}
  \begin{aligned}
    \frac{1}{\ell(A_{n})}\beta_{q}^{r, s}(\K^{\circ}(\Phi_{\underA_n^{(M)}}))
     & =\frac{1}{\ell(A_{n})}\sum_{z\in M\Z^{d}\cap\underA_n^{(M)}}\beta_{q, \Lambda_{M}+z}^{r, s} \\
     & =\frac{1}{\ell(A_{n})}\sum_{z\in M\Z^{d}\cap\underA_n^{(M)}}
    \beta_{q}^{r, s}(\K((T_{-z\ast}\Phi))_{\Lambda_{M}})
    \to \frac{1}{M^{d}}\E[\beta_{q, \Lambda_{M}}^{r, s}],
  \end{aligned}
\end{equation}
and
\begin{equation}\label{conv2}
  \begin{aligned}
    \frac{1}{\ell(A_{n})}\sum_{z\in
      M\Z^{d}\cap\underA_n^{(M)}}S_{j}(\Phi_{\partial \Lambda_{M}^{2\rho(s)}+z}, s)
     & =\frac{1}{\ell(A_{n})}\sum_{z\in
      M\Z^{d}\cap\underA_n^{(M)}}S_{j}((T_{-z\ast}\Phi)_{\partial
    \Lambda_{M}^{2\rho(s)}}, s)         \\
     & \to
    \frac{1}{M^{d}}\E[S_{j}(\Phi_{\partial \Lambda_{M}^{2\rho(s)}}, s)]
  \end{aligned}
\end{equation}
as $n\to \infty$ a.s.\ for any $M>0$. If we notice the fact that
$S_{j}(\Phi_{A}, s)+S_{j}(\Phi_{B}, s)\le
  S_{j}(\Phi_{A\cup B}, s)$ holds for disjoint bounded $A, B\in \cB(\R^{d})$, we see from Lemma \ref{lem:conv of Sq} that
\begin{equation}\label{conv3}
  \frac{1}{\ell(A_{n})}S_{j}(\Phi_{A_{n}\setminus\underA_n^{(M)}}, s)
  \le \left|\frac{1}{\ell(A_{n})}S_{j}(\Phi_{A_{n}}, s)
  -\frac{1}{\ell(A_{n})}S_{j}(\Phi_{\underA_n^{(M)}}, s)
  \right|\to 0
\end{equation}
as $n\to \infty$ a.s.\ for any $M>0$. Hence we can find $\Omega_{0}\in \cF$ with
$\P(\Omega_{0})=1$ such that for any $\omega\in \Omega_{0}$ and positive integer
$M$, the convergences (\ref{conv1}), (\ref{conv2}), and (\ref{conv3}) hold
as $n\to \infty$.
Take any $\omega\in \Omega_{0}$ and $\varepsilon>0$. If we choose a positive integer $M$ so that the inequalities (\ref{inq:M}) hold,
we have
\[
  \begin{aligned}
     & \left|\frac{1}{\ell(A_{n})}\beta_{q, A_{n}}^{r, s}(\omega)-\hat{\beta}_{q}^{r, s}\right|                                                                        \\
     & \ \le \frac{1}{\ell(A_{n})}\left|\beta_{q, A_{n}}^{r, s}(\omega)
    -\beta_{q}^{r, s}(\K^{\circ}(\Phi_{\underA_n^{(M)}}(\omega)))\right|
    +\left|\frac{1}{\ell(A_{n})}\beta_{q}^{r, s}(\K^{\circ}(\Phi_{\underA_n^{(M)}}(\omega)))-\frac{1}{M^{d}}\E[\beta_{q, \Lambda_{M}}^{r, s}]\right|                   \\
     & \ \ +\left|\frac{1}{M^{d}}\E[\beta_{q, \Lambda_{M}}^{r, s}]-\hat{\beta}_{q}^{r, s}\right|                                                                       \\
     & \le \sum_{j=q,
      q+1}\left(\frac{1}{\ell(A_{n})}\sum_{z\in
      M\Z^{d}\cap\underA_n^{(M)}}S_{j}(\Phi_{\partial \Lambda_{M}^{2\rho(s)}+z}(\omega), s)
    +\frac{1}{\ell(A_{n})}S_{j}(\Phi_{A_{n}\setminus\underA_n^{(M)}}(\omega), s)\right)                                                                                \\
     & \ +\left|\frac{1}{\ell(A_{n})}\beta_{q}^{r, s}(\K^{\circ}(\Phi_{\underA_n^{(M)}}(\omega)))-\frac{1}{M^{d}}\E[\beta_{q, \Lambda_{M}}^{r, s}]\right|+\varepsilon.
  \end{aligned}
\]
Consequently, we obtain
\[
  \begin{aligned}
    \limsup_{n\in \cN}\left|\frac{1}{\ell(A_{n})}\beta_{q, A_{n}}^{r, s}(\omega)-\hat{\beta}_{q}^{r, s}\right|
     & \le \sum_{j=q,
      q+1}\frac{1}{M^{d}}\E[S_{j}(\Phi_{\partial
    \Lambda_{M}^{2\rho(s)}}, s)]+\varepsilon                     \\
     & \le \sum_{j=q, q+1}\frac{C_{j
    ,s}\ell(\partial \Lambda_{M}^{2\rho(s)})}{M^{d}}+\varepsilon \\
     & \le 2\varepsilon,
  \end{aligned}
\]
which implies that the second assertion is valid.
The proof of Theorem \ref{th:LLN for PBNs} is now complete. \qed

\subsection{Convergence of persistence diagrams}\label{sec:convergence of PDs}

In this section we prove Theorem \ref{th:LLN for PDs}.
To this end, we make use of similar arguments which can be
found in the proof of the same kind of limit theorem for
persistence diagrams built over
stationary point process (Theorem 1.5 in \cite{HST}).
Let $X$ be a locally compact Hausdorff space with countable base
and $\cC$ the ring of all relatively compact sets in $X$.
A class $\cC^{\prime}\subset \cC$ is called a convergence-determining class
(for vague convergence) if for any $\mu\in \cM_{X}^{\#}$ and any sequence
$\{\mu_{n}\}\subset \cM_{X}^{\#}$, the condition
\[
  \mu_{n}(A)\to \mu(A) \text{ as } n\to \infty
  \text{ for all } A\in \cC^{\prime}\cap \cC_{\mu}
\]
implies the vague convergence $\mu_{n}\xrightarrow{v}\mu$, where
$\cC_{\mu}$ is the class of relatively compact continuity sets of $\mu$,
i.e., $\cC_{\mu}=\{B\in \cC\,:\,\mu(\partial B)=0\}$.
A class $\cC^{\prime}_{\mu}$ is called a convergence-determining class
for $\mu\in \cM_{X}^{\#}$ if for any sequence
$\{\mu_{n}\}\subset \cM_{X}^{\#}$, the condition
\[
  \mu_{n}(A)\to \mu(A) \text{ as } n\to \infty
  \text{ for all } A\in \cC^{\prime}_{\mu}
\]
implies the vague convergence $\mu_{n}\xrightarrow{v}\mu$.
We note that a class $\cC^{\prime}$ is a convergence-determining class
if and only if for any $\mu\in \cM_{X}^{\#}$,
$\cC^{\prime}\cap \cC_{\mu}$ is a convergence-determining class
for $\mu$.
A convergence-determining class $\cC^{\prime}$ has the finite covering property if
for any $B\in \cC$, $B$ is covered by a finite union of $\cC^{\prime}$-sets.
The next lemma can be proved in the same way as  Proposition 3.4 in \cite{HST}.

\begin{lemma}\label{key lemma1}
  Let $X$ be a locally compact Hausdorff space with countable base
  and $\cC^{\prime}$ a convergence-determining class with finite covering property.
  Suppose that for every $\mu\in \cM_{X}^{\#}$, $\cC^{\prime}$ contains a countable convergence-determining class for $\mu$. Let $\{\xi_{n}\}$ be
  a net of random measures on $X$ satisfying the following:
  \begin{itemize}
    \item[(i)]
          For every $n$, $\E[\xi_{n}]\in \cM_{X}^{\#}$.

    \item[(ii)]
          For every $A\in \cC^{\prime}$, there exists $c_{A}\ge 0$ such that
          $\E[\xi_{n}(A)]\to c_{A}$ as $n\to \infty$.
  \end{itemize}
  Then there exists a unique measure $\mu\in \cM_{X}^{\#}$ such that
  $\E[\xi_{n}]\xrightarrow{v}\mu$ as $n\to \infty$ and $\mu(A)=c_{A}$.
  Furthermore, if $\xi(A)\to c_{A}$ as $n\to \infty$ almost surely
  for any $A\in \cC^{\prime}$, then $\xi_{n}\xrightarrow{v} \mu$
  as $n\to \infty$ almost surely.
\end{lemma}

An example of convergence-determining classes satisfying the conditions in Lemma \ref{key lemma1} is the following.
\begin{lemma}[Corollary A.3 in \cite{HST}]\label{key lemma2}
  The class
  \[
    \cC^{\prime}=\{(r_{1}, r_{2}]\times(s_{1}, s_{2}],
    [0, r_{2}]\times (s_{1}, s_{2}]\subset\Delta\,:\,
    0\le r_{1}\le r_{2}\le s_{1}\le s_{2}\le \infty\}
  \]
  is a convergence-determining class which satisfies the
  conditions in Lemma \ref{key lemma1}.
\end{lemma}

We finish with the proof of Theorem \ref{th:LLN for PDs}.

\textit{Proof of Theorem \ref{th:LLN for PDs}}.
Suppose that $R$ is a rectangle of the form $(r_{1}, r_{2}]\times (s_{1}, s_{2}]$ or $[0, r_{1}]\times (s_{1}, s_{2}]$ in $\Delta$.
By virtue of Lemma \ref{key lemma1} and Lemma \ref{key lemma2},
we have only to show that there exists $c_{R}\ge 0$ such that for any convex averaging net $\cA=\{A_{n}\}_{n\in \cN}$,
\[
  \frac{1}{\ell(A_{n})}\E[\xi_{q, A_{n}}(R)]\to c_{R} \quad \text{ as } n\to \infty
\]
and if $\Phi$ is ergodic, then
\[
  \frac{1}{\ell(A_{n})}\xi_{q, A_{n}}(R)\to c_{R} \quad \text{ as } n\to \infty \text{ a.s. }
\]
It follows immediately from Theorem \ref{th:LLN for PBNs} and the fact that
$\xi_{q, A_{n}}(R)$ is calculated as
\[
  \begin{aligned}
    \xi_{q, A_{n}}(R)
     & =\xi_{q, A_{n}}([0, r_{2}]\times(s_{1},  \infty])
    -\xi_{q, A_{n}}([0, r_{2}]\times(s_{2},  \infty])                \\
     & \hspace{12pt}+\xi_{q, A_{n}}([0, r_{1}]\times(s_{2}, \infty])
    -\xi_{q, A_{n}}([0, r_{1}]\times(s_{1}, \infty])                 \\
     & =\beta_{q, A_{n}}^{r_{2}, s_{1}}
    -\beta_{q, A_{n}}^{r_{2}, s_{2}}
    +\beta_{q, A_{n}}^{r_{1}, s_{2}}
    -\beta_{q, A_{n}}^{r_{1}, s_{1}}
  \end{aligned}
\]
for $R=(r_{1}, r_{2}]\times (s_{1}, s_{2}]$ and
\[
  \xi_{q, A_{n}}(R)
  =\beta_{q, A_{n}}^{r_{1}, s_{1}}
  -\beta_{q, A_{n}}^{r_{1}, s_{2}}
\]
for $R=[0, r_{1}]\times (s_{1}, s_{2}]$. Thus we arrive at the desired result. \qed

\section*{Acknowledgment}
This work was supported by JST CREST Grant Number
JPMJCR15D3, Japan.
The first named author (T.S.) was supported by
Grant-in-Aid for Exploratory Research (JP17K18740),
the Grant-in-Aid for Scientific Research (B) (JP18H01124) and (S) (JP16H06338) of Japan Society for the Promotion of
Science.



\end{document}